\documentclass[12pt,reqno]{amsart}
\usepackage{amsmath}
\usepackage{amssymb}
\usepackage[left=3cm,top=3cm,right=3cm,bottom=3cm]{geometry}
\usepackage{epsfig}

\begin{document}
\newtheorem{theorem}{Theorem}[section]
\newtheorem{lemma}[theorem]{Lemma}
\newtheorem{definition}[theorem]{Definition}
\newtheorem{conjecture}[theorem]{Conjecture}
\newtheorem{proposition}[theorem]{Proposition}
\newtheorem{algorithm}[theorem]{Algorithm}
\newtheorem{corollary}[theorem]{Corollary}
\newtheorem{observation}[theorem]{Observation}
\newtheorem{problem}[theorem]{Open Problem}
\newcommand{\noin}{\noindent}
\newcommand{\ind}{\indent}
\newcommand{\al}{\alpha}
\newcommand{\om}{\omega}
\newcommand{\pp}{\mathcal P}
\newcommand{\ppp}{\mathfrak P}
\newcommand{\R}{{\mathbb R}}
\newcommand{\N}{{\mathbb N}}
\newcommand\eps{\varepsilon}
\newcommand{\E}{\mathbb E}
\newcommand{\Prob}{\mathbb{P}}
\newcommand{\pl}{\textrm{C}}
\newcommand{\dang}{\textrm{dang}}
\renewcommand{\labelenumi}{(\roman{enumi})}
\newcommand{\bc}{\bar c}
\newcommand{\G}{{\mathcal{G}}}
\newcommand{\expect}[1]{\E \left [ #1 \right ]}
\newcommand{\ceil}[1]{\left \lceil #1 \right \rceil}
\newcommand{\of}[1]{\left( #1 \right)}
\newcommand{\set}[1]{\left\{ #1 \right\}}
\newcommand{\size}[1]{\left \vert #1 \right \vert}
\newcommand{\floor}[1]{\left \lfloor #1 \right \rfloor}

\title{Lazy Cops and Robbers played on Graphs}

\author{Deepak Bal}
\address{Department of Mathematics, Ryerson University, Toronto, ON, Canada, M5B 2K3}
\email{deepak.c.bal@ryerson.ca}

\author{Anthony Bonato}
\address{Department of Mathematics, Ryerson University, Toronto, ON, Canada, M5B 2K3}
\email{\tt abonato@ryerson.ca}

\author{William B.\ Kinnersley}
\address{Department of Mathematics, Ryerson University, Toronto, ON, Canada, M5B 2K3}
\email{\tt wkinners@ryerson.ca}

\author{Pawe{\l} Pra{\l}at}
\address{Department of Mathematics, Ryerson University, Toronto, ON, Canada, M5B 2K3}
\email{\tt pralat@ryerson.ca}

%%% check/update it
\thanks{Supported by grants from NSERC and Ryerson}
\keywords{Cops and Robbers, vertex-pursuit games, random graphs, domination, adjacency property, planar graphs, hypercubes}
\subjclass{05C57, 05C80}

\maketitle

\begin{abstract}
We consider a variant of the game of Cops and Robbers, called Lazy Cops and Robbers, where at most one cop can move in any round.  We investigate the analogue of the cop number for this game, which we call the lazy cop number.  Lazy Cops and Robbers was recently introduced by Offner and Ojakian, who provided asymptotic upper and lower bounds on the lazy cop number of the hypercube.  By investigating expansion properties, we provide asymptotically almost sure bounds on the lazy cop number of binomial random graphs $\G(n,p)$ for a wide range of $p=p(n)$. By coupling the probabilistic method with a potential function argument, we also improve on the existing lower bounds for the lazy cop number of hypercubes. Finally, we provide an upper bound for the lazy cop number of graphs with genus $g$ by using the Gilbert-Hutchinson-Tarjan separator theorem.
\end{abstract}

\section{Introduction}

The game of Cops and Robbers (defined, along with all the standard notation, at the end of this section) is usually studied in the context of the {\em cop number}, the minimum number of cops needed to ensure a winning strategy. The cop number is often challenging to analyze; establishing upper bounds for this parameter is the focus of Meyniel's conjecture that the cop number of a connected $n$-vertex graph is $O(\sqrt{n}).$ For additional background on Cops and Robbers and Meyniel's conjecture, see the book~\cite{bonato} and the surveys~\cite{bonato1,bonato2,bonato3}.

A number of variants of Cops and Robbers have been studied. For example, we may allow a cop to capture the robber from a distance $k$, where $k$ is a non-negative integer~\cite{bonato5,bonato4}, play on edges~\cite{pawel}, allow one or both players to move with different speeds~\cite{NogaAbbas, fkl} or to teleport, allow the robber to capture the cops~\cite{bonato0}, or make the robber invisible or drunk~\cite{drunk1,drunk2}. See Chapter~8 of~\cite{bonato} for a non-comprehensive survey of variants of Cops and Robbers.

We are interested in slowing the cops down to create a situation akin to chess, where at most one chess piece can move in a round. Hence, our focus in the present article is a recent variant introduced by Offner and Ojakian~\cite{oo}, where at most one cop can move in any given round. We refer to this variant, whose formal definition appears in Section \ref{sec:defns}, as {\em Lazy Cops and Robbers}; the analogue of the cop number is called the {\em lazy cop number}, and is written $c_L(G)$ for a graph $G.$  In~\cite{oo} it was proved for the hypercube $Q_n$ that $2^{\floor{\sqrt{n}/20}} \le c_L(Q_n) \le O(2^n \log n/n^{3/2})$.  We mention in passing that \cite{oo} introduced a number of variants of Cops and Robbers, in which some fixed number of cops (perhaps more than one) can move in a given round. We focus here on the extreme case in which only one cop moves in each round, but it seems likely that our techniques generalize to other variants. 

We present three results on Lazy Cops and Robbers and the lazy cop number. In Theorem~\ref{thm:main_gnp} we provide asymptotically almost sure bounds on the lazy cop number for binomial random graphs $\G(n,p)$ for a wide range of $p=p(n)$. We do this by examining typical expansion properties of such graphs. In Theorem~\ref{thm:hyp-lower}, by using the probabilistic method coupled with a potential function argument, we improve on the lower bound for the lazy cop number of hypercubes given in \cite{oo}. In Theorem~\ref{thm:bded_genus}, we provide an upper bound for graphs of genus $g$ using the Gilbert-Hutchinson-Tarjan separator theorem~\cite{ght}.

\subsection{Definitions and notation}\label{sec:defns}
We consider only finite, undirected graphs in this paper. For background on graph theory, the reader is directed to~\cite{west}.

The game of \emph{Cops and Robbers} was independently introduced in~\cite{nw,q} and the cop number was introduced in~\cite{af}. The game is played on a reflexive
graph; that is, each vertex has at least one loop. Multiple edges are
allowed, but make no difference to the play of the game, so we always assume there
is exactly one edge joining adjacent vertices. There are two players,
consisting of a set of \emph{cops} and a single \emph{robber}. The game is
played over a sequence of discrete time-steps or \emph{turns},
with the cops going first on turn $0$ and then playing on alternate time-steps.  A \emph{round} of the game is a cop move together with the subsequent robber move.  The cops and robber occupy vertices;
for simplicity, we often identify the player with the vertex they occupy. We
refer to the set of cops as $C$ and the robber as $R.$ When a player is ready to move in a round they must
move to a neighbouring vertex. Because of the loops, players can \emph{pass}, or remain on their own vertices. Observe that any subset of $C$ may
move in a given round. The cops win if after some finite number of rounds, one of them can occupy
the same vertex as the robber (in a reflexive graph, this is equivalent to the cop landing on the robber).
This is called a \emph{capture}. The robber
wins if he can
evade capture indefinitely. A \emph{winning strategy for the cops} is a set
of rules that if followed, result in a win for the cops. A \emph{winning
strategy for the robber} is defined analogously.  As stated earlier, the game of \emph{Lazy Cops and Robbers} is defined almost exactly as Cops and Robbers, with the exception
that exactly one cop moves in any round.

If we place a cop at each vertex, then the cops are guaranteed to win.
Therefore, the minimum number of cops required to win in a graph $G$ is a
well-defined positive integer, named the \emph{lazy cop number} of the graph $G.$ We write $c_L(G)$ for the lazy cop number of a graph $G$.

\section{Random graphs}

In this section, we consider the game played on binomial random graphs. The \emph{random graph} $\G(n,p)$ consists of the probability space $(\Omega, \mathcal{F}, \Prob)$, where $\Omega$ is the set of all graphs with vertex set $\{1,2,\dots,n\}$, $\mathcal{F}$ is the family of all subsets of $\Omega$, and for every $G \in \Omega$,
$$
\Prob(G) = p^{|E(G)|} (1-p)^{{n \choose 2} - |E(G)|} \,.
$$
This space may be viewed as the set of outcomes of ${n \choose 2}$ independent coin flips, one for each pair $(u,v)$ of vertices, where the probability of success (that is, adding edge $uv$) is $p.$ Note that $p=p(n)$ may tend to zero as $n$ tends to infinity. All asymptotics throughout are as $n \rightarrow \infty $ (we emphasize that the notations $o(\cdot)$ and $O(\cdot)$ refer to functions of $n$, not necessarily positive, whose growth is bounded). We say that an event in a probability space holds \emph{asymptotically almost surely} (or \emph{a.a.s.}) if the probability that it holds tends to $1$ as $n$ goes to infinity.

\bigskip

Let us first briefly describe some known results on the (non-lazy) cop number of $\G(n,p)$. Bonato, Wang, and Pra\l{}at investigated such games in $\G(n,p)$ random graphs and in generalizations used to model complex networks with power-law degree distributions (see~\cite{bpw}). From their results it follows that if $2 \log n / \sqrt{n} \le p < 1-\eps$ for some $\eps>0$, then a.a.s. we have that
\begin{equation}\label{eq:classig_c}
c(\G(n,p))= \Theta(\log n/p),
\end{equation}
so Meyniel's conjecture holds a.a.s.\ for such $p$. In fact, for $p=n^{-o(1)}$ we have that a.a.s.\  $c(\G(n,p))=(1+o(1)) \log_{1/(1-p)} n$. A simple argument using dominating sets shows that Meyniel's conjecture also holds a.a.s.\  if $p$ tends to 1 as $n$ goes to infinity (see~\cite{p} for this and stronger results). Bollob\'as, Kun and Leader~\cite{bkl} showed that if $p(n) \ge 2.1 \log n /n$, then a.a.s.
$$
\frac{1}{(pn)^2}n^{ 1/2 - 9/(2\log\log (pn))  }  \le c(\G(n,p))\le 160000\sqrt n \log n\,.
$$
From these results, if $np \ge 2.1 \log n$ and either $np=n^{o(1)}$ or $np=n^{1/2+o(1)}$, then a.a.s.\ $c(\G(n,p))= n^{1/2+o(1)}$. Somewhat surprisingly, between these values it was shown by \L{}uczak and Pra\l{}at~\cite{lp2} that the cop number has more complicated behaviour. It follows that a.a.s.\ $\log_n  c(\G(n,n^{x-1}))$ is asymptotic to the function (denoted in blue) shown in Figure~\ref{fig1}.
\begin{figure}[h]
\begin{center}
\includegraphics[width=3.6in]{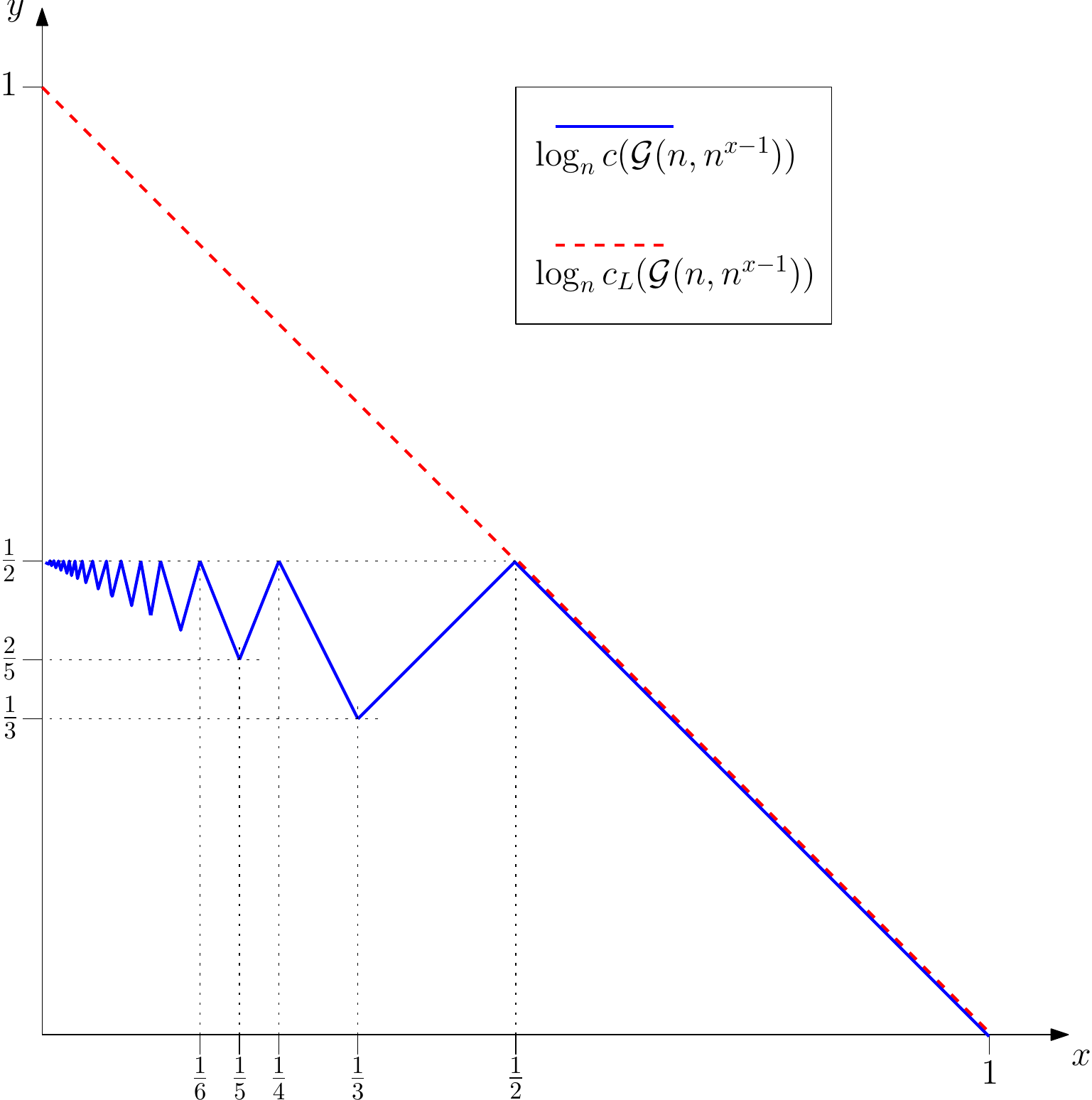}
\end{center}
\caption{The ``zigzag'' function representing the ordinary cop number along with the function representing the lazy cop number.}\label{fig1}
\end{figure}

The above results show that Meyniel's conjecture holds a.a.s.\ for random graphs except perhaps when $np=n^{1/(2k)+o(1)}$ for some $k \in \N$, or when $np=n^{o(1)}$. Pra\l{}at and Wormald showed recently that the conjecture holds a.a.s.\ in $\G(n,p)$~\cite{PW_gnp} as well as in random $d$-regular graphs~\cite{PW_gnd}.

\bigskip

In this paper, we investigate the lazy cop number of $\G(n,p)$. The main theorem of this section is the following.

\begin{theorem}\label{thm:main_gnp}
Let $0<\alpha \le1$, let $\eps>0$, and let $d=d(n)=(n-1)p=n^{\alpha+o(1)}$.
\begin{enumerate}
\item If $\alpha = 1$ and $p < 1-\eps$, then a.a.s.\
$$
c_L(\G(n,p))= (1+o(1)) \log_{1/(1-p)} n\,.
$$
(Note that if $p=o(1)$, then $\log_{1/(1-p)} n = (1+o(1)) \frac {\log n}{p}$.)
\item If $\frac{1}{2}<\alpha<1$, then a.a.s.\
$$
c_L(\G(n,p))= \Theta \left( \frac {\log n}{p} \right)\,.
$$
\item If $\frac{1}{j+1}<\alpha<\frac{1}{j}$ for some integer $j \ge 2$, then a.a.s.\
$$
 \frac{1}{p} =O\big( c_L(\G(n,p))\big)= O \left( \frac{\log n}{p} \right)\,.
$$
\item If $\alpha=\frac{1}{j}$ for some integer $j \ge 2$, then a.a.s.\
$$
\frac{1}{p \log^2 n} =O\big( c_L(\G(n,p))\big)= O \left( \frac{\log n}{p} \right)\,.
$$
\end{enumerate}
In particular, a.a.s.\ $c_L(\G(n,p))=n^{1-\alpha+o(1)}$.
\end{theorem}

See Figure~\ref{fig1} for corresponding function (denoted in dotted red) for the lazy cop number. In fact, for case (iv) we prove a slightly stronger lower bound---see Theorem~\ref{thm:gnp_lower} for more details.

\subsection{Upper bound}

First, let us note that $c_L(G) \le \gamma(G)$ for all graphs $G$, since by initially occupying a dominating set of $G$, the cops win on their first turn. Moreover, it is well-known (and straightforward to show using, for example, the probabilistic method) that for any graph $G$
$$
\gamma(G) \le n \frac {\log(\delta+1)}{\delta+1},
$$
where $\delta=\delta(G)$ is the minimum degree of $G$. For $G \in \G(n,p)$, when $pn \gg \log n$ we have that a.a.s.\ $\delta(G)=(1+o(1)) pn$.  Consequently, a.a.s.
$$
c_L(G) \le (1+o(1)) \frac {\log (pn)}{p},
$$
provided that $pn \gg \log n$; this provides the upper bound for cases (ii)-(iv) in Theorem~\ref{thm:main_gnp}. When $p=\Omega(1)$ but $p<1-\eps$ for some $\eps>0$ (see case (i) of Theorem~\ref{thm:main_gnp}), one can easily show that a.a.s.\ 
$$
\gamma(G) \le \log_{1/(1-p)} n + \log_{1/(1-p)} \omega = (1+o(1)) \log_{1/(1-p)} n,
$$  
where $\omega = \omega(n)=n^{o(1)}$ is any function tending to infinity sufficiently slowly as $n \to \infty$. Indeed, any set of vertices with cardinality $\log_{1/(1-p)} n + \log_{1/(1-p)} \omega = (1+o(1)) \log_{1/(1-p)} n$ is  a.a.s. a dominating set.

\subsection{Lower bound}

For dense graphs (cases (i)-(ii) in Theorem~\ref{thm:main_gnp}) it is enough to use results for the classic cop number (see~(\ref{eq:classig_c}) and subsequent discussion) and the trivial observation that $c_L(G) \ge c(G)$.

For sparse graphs (cases (iii)-(iv) in Theorem~\ref{thm:main_gnp}), let us start by proving some typical properties of $\G(n,p)$. These observations are part of folklore, but here we provide all proofs for completeness.

Let $N_i[v]$ denote the set of vertices within distance $i$ of $v$. For simplicity, we use $N[v]$ to denote $N_1[v]$. Moreover, let $N[S] = \bigcup_{v \in S} N[v]$. Finally, let $P_i(v,w)$ denote the number of paths of length $i$ joining $v$ and $w$.

\begin{lemma}\label{l:elem}
Let $\eps$ and $\alpha$ be constants such that $0<\eps<0.1$, $\eps<\alpha<1-\eps$, and let $d=d(n)=p(n-1)=n^{\alpha+o(1)}$. Then a.a.s.\ for every vertex $v$ of $G=(V,E) \in \G(n,p)$ the following properties hold.
\begin{enumerate}
\item For every $i$ such that $d^i = o(n)$, we have
$$
|N_i[v]| = (1+o(1)) d^i.
$$
Furthermore, for $d^i = cn$ with $c = c(n) \le 1$ and $c = \Omega(1)$,
$$
|N_i[v]| = (1-e^{-c}+o(1)) d^i.
$$

\item Let $\ell \in \N$ be the largest integer such that $\ell < 1/\alpha$.  Then the following hold:
\begin{enumerate}
\item If $w\in N_i[v]$ for some $i$ with $2\le i \le \ell$, then $P_i(v,w) \le 3/(1-i\alpha)$.
\item If $w\in N_{\ell+1}[v]$ and $d^{\ell+1} \ge 7 n \log n$, then $P_{\ell+1}(v,w) \le \frac{6}{1-\ell \alpha}\frac{d^{\ell+1}}{n}$.
\item If $w\in N_{\ell+1}[v]$ and $d^{\ell+1} < 7 n \log n$, then $P_{\ell+1}(v,w) \le \frac{42}{1-\ell \alpha} \log n$.
\item If $w\in N_{\ell+2}[v]$ and $d^{\ell+1} < n$, then $P_{\ell+2}(v,w) \le \frac{84}{1-\ell \alpha} \frac {d^{\ell+2}\log n}{n}$.
\end{enumerate}
\item If $i$ satisfies $d^i < n/\log n$, then every edge of $G$ is contained in at most $\eps d$ cycles of length at most $i+2$.
\end{enumerate}
\end{lemma}
\begin{proof}
Let $S \subseteq V$, let $s=|S|$, and consider the random variable \[X = X(S) = |\left\{ v\in V\setminus S\,:\, uv \in E \text{ for some } u\in S \right\}|,\] that is, the number of vertices outside of $S$ and adjacent to at least one vertex in $S$.  For (i), we bound $X$ in a stochastic sense. There are two things that need to be estimated: the expected value of $X$ and the concentration of $X$ around its expectation.

It is evident that 
\begin{eqnarray*}
\E [X] &=& \left( 1 - \left(1- \frac {d}{n-1} \right)^s \right) (n-s) \\
&=& \left( 1 - \exp \left( - \frac {ds}{n} (1+O(d/n)) \right) \right) n (1+O(s/n)) \\
&=& ds (1+O(ds/n)) \\
&=& ds (1+O(\log^{-1} n)),
\end{eqnarray*}
provided $ds \le n/ \log n $. We next use a consequence of Chernoff's bound (see e.g.~\cite[p.\ 27\ Cor.\ 2.3]{JLR}), that
\begin{equation}\label{chern}
\Prob( |X-\E [X]| \ge \eps \E [X]) \le 2\exp \left( - \frac {\eps^2 \E [X]}{3} \right)
\end{equation}
for  $0 < \eps < 3/2$. This implies that the expected number of sets $S$ such that $\big| X(S) - d|S| \big| > \eps d|S|$ and $|S| \le n/(d\log n)$  is, for $\eps = 2/{\log n}$, at most
\begin{equation}\label{badS}
\sum_{s=1}^{n/(d\log n)} \binom{n}{s} \cdot 2 \exp \left( - \frac {\eps^2 ds}{3+o(1)} \right)\le \sum_{s=1}^{n/(d\log n)} 2n^s \exp \left( - \frac {\eps^2 s \log^3 n}{3+o(1)} \right) = o(1),
\end{equation}
where the first inequality uses the fact that $d\ge\log^3 n$.

So a.a.s.\ if $ |S| \le n/d\log n$, then $X(S) =d |S| (1+O(1/\log n))$, where the bound in $O()$ is uniform.  Since $d \ge \log^3 n$, for such sets we have 
$$
|N[S]| = \size{S} + X(S)= X(S) (1+O(1/d)) = d |S| (1+O(1/\log n)).
$$
We may assume this equation holds deterministically.

This assumption yields good bounds on the ratios of $\size{N[v]}$ and $\size{N[N[v]]} = \size{N_2[v]}$, of $\size{N_2[v]}$ and $\size{N_3[v]}$, and so on. These bounds apply to the ratio $\size{N_r[v]} / \size{N_{r-1}[v]}$ so long as $d^r \le {n}/{\log n}$.  The cumulative multiplicative error across these ratios is $(1+O(\log^{-1}n))^r$, which is $(1+o(1))$ since $r$ can be at most $1/\alpha + o(1) = O(1)$.  Thus,
\begin{equation}\label{eq:i}
| N_r [v] | = (1+o(1)) d^{r}
\end{equation}
for all vertices $v$ and $r$ such that $d^r \le n / \log n$, which establishes (i) in this case.

Suppose now that $d^r = cn$ with $c = c(n) \le 1$ and $c = \Omega(1)$.  Let $U = N_{r-1}[v]$.  Using~(\ref{eq:i}), we have that $\size{U} = (1+o(1))d^{r-1}$, so applying the assumption that $d^r = cn$, we have
\begin{align*}
\expect{\size{N[U]}} &= \size{U} + \of{1 - (1-p)^{\size{U}}}(n-\size{U})\\
 &=|U| + \of{1 - \exp\of{-\frac{d\cdot d^{r-1}(1+o(1))}{n}\of{1 + O(d/n)}}}(n- |U|)\\
 &= n(1-e^{-c} + o(1)).
\end{align*}
Chernoff's bound~(\ref{chern}) can be used again, in the same way as before, to show that in this case a.a.s.\ $|N[U]|$ is concentrated near its expected value for all $v$ and $r$. Thus, (i) holds also in this case.

\bigskip

The first part of (ii) can be easily verified using the first moment method. Indeed, suppose there exists $w\in N_i[v]$, for some $v \in V$ and $i \in \N$ with $2\le i<1/\alpha$, such that $v$ and $w$ are joined by $k = \lceil 3/(1-i\alpha) \rceil$ internally disjoint paths of length $i$. This structure has $2+k(i-1)$ vertices and $ki$ edges. The expected number of such subgraphs in $G$ is
$$
O(n^{2+k(i-1)} p^{ki}) \le n^{2+k(i-1)+ki(\alpha-1)+o(1)} = n^{2-k(1-i\alpha)+o(1)} \le n^{-1+o(1)} = o(1).
$$
Hence, a.a.s.\ there is no such subgraph in $G$. Since all other possible structures joining $v$ and $w$ by $k = \lceil 3/(1-i\alpha) \rceil$ paths of length $i$ (not necessarily internally disjoint) are even denser, the same argument applies to them as well. Finally, since $\alpha$ is constant and so is $i \le 1/\alpha + o(1)$, there are only finitely many structures to consider. The claim follows by the union bound.  

The second part of (ii) is a consequence of (i), the first part of (ii), and Chernoff's bound. Suppose $d^{\ell+1} \ge 7n \log n$. Let us first expose the $\ell$th neighbourhood of $v$.  By (i), we may assume that $|N_\ell[v]|=(1+o(1)) d^{\ell}$. For any $w \in V \setminus N_\ell[v]$, the expected number of edges joining $w$ to $N_\ell[v]$ is $p |N_\ell[v]|= (1+o(1)) d^{\ell+1}/n \ge (7+o(1)) \log n$. It follows from~(\ref{chern}) that with probability $1-o(n^{-2})$ there are at most $2d^{\ell+1}/n$ edges joining $w$ to $N_\ell[v]$. By the first part of (ii), we may assume that every vertex is joined to $v$ by fewer than $3/(1-\ell \alpha)$ paths of length $\ell$. Hence, with probability $1-o(n^{-2})$, the desired bound on the number of $v,w$-paths of length $\ell+1$ holds for the pair $v,w$.  The desired result holds by applying the union bound over all pairs under consideration.

Suppose now that $d^{\ell+1} < 7n \log n$. This time, the expected number of edges joining $w$ and $N_\ell[v]$ is at most $(7+o(1)) \log n$, and we apply the more common form of Chernoff's bound: if $X$ is distributed as ${\rm Bin}(n,p)$, then
\begin{equation}\label{chern2}
\Prob (X \ge \E [X] + t) \le \exp \left( - \frac {t^2}{2(\E [X]+t/3)} \right).
\end{equation}
This shows that with probability $1-o(n^{-2})$, there are at most $14 \log n$ edges joining $w$ and $N_\ell[v]$. The rest of the argument works as before.

Finally, suppose $d^{\ell+1} < n$. By (i), we may assume that
$$
(1-e^{-1}+o(1)) d^{\ell+1} \le |N_{\ell+1}[v]| \le (1+o(1)) d^{\ell+1}.
$$
The expected number of edges joining some $w \in V \setminus N_{\ell+1}[v]$ to $N_{\ell+1}[v]$ is $p |N_{\ell+1}[v]| = \Theta (d^{\ell+2}/n) \ge (7+o(1)) \log n$  It follows from~(\ref{chern}) that with probability $1-o(n^{-2})$ there are at most $2d^{\ell+2}/n$ edges joining $w$ to $N_{\ell+1}[v]$, and the desired bound holds.

\bigskip

In order to verify (iii) it is enough to check that for any given pair of vertices $r_1$ and $r_2$, and any $j$ such that $2\le j\le i+1$, the probability that $G$ contains more than $\eps d /(2i)$ different $r_1,r_2$-paths of length $j$ is $o(n^{-2})$; the result then follows by applying the union bound over all adjacent pairs $(r_1,r_2)$. Denote the number of such paths by $X_j^{r_1,r_2}(n,p)$. For the expectation of $X_j^{r_1,r_2}(n,p)$, we have
\begin{eqnarray*}
\E [X_j^{r_1,r_2}(n,p)] &=&\binom {n-2}{j-1}(j-1)! p^{j} \\
&<& n^{j-1}p^j \le \frac{d^{j}}{n} \\
&\le & d \frac{d^i}{n}\le d\frac{1}{\log n} \\&\le &\frac{\eps d}{4i}\,.
\end{eqnarray*}

Now, choose $p'>p$ in such a way that
$$
\E [X_j^{r_1,r_2}(n,p')]=\frac{\eps d}{4i}\,
$$
and note that this is at least $\log^{2(j-1)}n$.
Then by a result of Vu (see~\cite{Vu}, Corollary 2.6) it follows that for some constant $a>0$,
\begin{eqnarray*}
\Pr(X_j^{r_1,r_2}(n,p')>{\eps d}/({2i})) &\le& \Pr(X_j^{r_1,r_2}(n,p')>2\E  [X_j^{r_1,r_2}(n,p')]) \\
&\le& \exp\big(-a (\E [X_j^{r_1,r_2}(n,p')])^{1/(j-1)}\big) \\
&\le& \exp\big(-a \log^2 n)\big)=o(n^{-2})\,.
\end{eqnarray*}
Hence,
$$
\Pr\Big(X_j^{r_1,r_2}(n,p)>\frac{\eps d}{2i}\Big)\le \Pr\Big(X_j^{r_1,r_2}(n,p')>\frac{\eps d}{2i}\Big)=o(n^{-2})\,,
$$
and the assertion follows.
\end{proof}

Now, we are ready to prove our lower bound on $c_L(G)$ for $G \in \G(n,p)$. The proof is an adaptation of the proof used for the classic cop number in~\cite{lp2}. Let us point out that in this paper we also deal with the case $\alpha=1/(j+1)$, which was omitted in~\cite{lp2}.

\begin{theorem}\label{thm:gnp_lower}
Let $\frac{1}{j+1} < \alpha<\frac{1}{j}$ for some $j \in \N$, $c=c(j,\alpha)=\frac{6}{1-j \alpha}$, and $d=d(n)=p(n-1)=n^{\alpha+o(1)}$. Then a.a.s.\ for $G = (V,E) \in \G(n,p)$ we have that
\begin{equation}\label{eq:lower_gnp1}
c_L(G) \ge \frac {1-j\alpha}{12 (2c)^{j-1} j^j} \cdot \frac {1}{p} = \Omega \left( \frac 1p \right).
\end{equation}

Let $\alpha = \frac {1}{j+1}$ for some $j \in \N$. Then a.a.s.\ for $G = (V,E) \in \G(n,p)$ we have that
\begin{equation}\label{eq:lower_gnp2}
c_L(G) \ge
\begin{cases}
\frac {1-j\alpha}{12 (2c)^{j-1} j^j} \cdot \frac {1}{p} = \Omega \left( \frac 1p \right), & \text{ if } d^{j+1} \ge 7 n \log n\\
\frac{c(1-j\alpha)}{42 (2cj)^j} \cdot \frac{d^j}{\log n} = \Omega \left( \left( \frac {d^{j+1}}{n \log n} \right) \frac 1p \right) = \Omega \left( \frac {1}{p \log^2 n} \right), & \text{ if } n/\log n \le d^{j+1} < 7 n \log n\\
\frac{c^2(1-j\alpha)^2}{3528 (2c(j+1))^{j+1}} \cdot \frac{n}{d \log^2 n} = \Omega \left( \frac {1}{p \log^2 n} \right), & \text{ if } d^{j+1} < n/\log n.
\end{cases}
\end{equation}
\end{theorem}
\begin{proof}
In all cases, we provide a winning strategy for the robber on $G$.  Since our aim is to prove that the bounds hold a.a.s., we may assume without loss of generality that $G$ satisfies the properties stated in Lemma~\ref{l:elem}.

Suppose first that $1/(j+1) < \alpha < 1/j$ and that the robber is chased by $K=\frac {1-j\alpha}{12 (2c)^{j-1} j^j} \cdot \frac {1}{p}$ cops. For vertices $x_1, x_2, \dots,x_s$, let $\pl^{x_1, x_2, \dots,x_s}_i(v)$ denote the number of cops in the $i$th neighbourhood of $v$ in the graph induced by $V \setminus \{x_1,x_2,\dots,x_s\}$; in particular, if $v\notin\{x_1,x_2,\dots,x_s\}$, then $\pl^{x_1,x_2,\dots,x_s}_0(v)=0$ if and only if $v$ is not occupied by a cop. Right before the cops make their move, we say that the vertex $v$ occupied by the robber is {\em safe} if for some neighbour $x$ of $v$ we have $\pl^x_0(v)=0$, $\pl^x_1(v)=0$, and
$$
\pl_{i}^x(v) \le \Big(\frac{d}{2cj} \Big)^{i-1}
$$
for $i=2,3,\dots,j$; such a vertex $x$ will be called a {\em deadly neighbour} of $v$.

Since a.a.s.\ $G$ is connected, without loss of generality we may assume that at the beginning of the game all cops begin at the same vertex, $z$. Subsequently, the robber may choose a vertex $v$ at distance $j+1$ from $z$ (see Lemma~\ref{l:elem}(i) with $i=j$); clearly $v$ is safe. Hence, in order to prove the theorem, it is enough to show that if the robber's current vertex is safe, then no matter how the cops move in the next round, the robber can always move to a safe vertex.

For $r\ge 0$, we say that a neighbour $y\neq x$ of $v$ is {\em $r$-dangerous} if
\begin{itemize}
\item [(i)] $\pl^{v,x}_r(y)>0$ (for $r=0,1$)\,, or
\item [(ii)] $\pl^{v,x}_r(y)> \Big(\frac{d}{2cj} \Big)^{r-1}$ (for $r= 2,3, \ldots, j$)\,,
\end{itemize}
where $x$ is a deadly neighbour of $v$. The idea here is that $x$ represents the vertex occupied by the robber on the previous turn.  Clearly, excluding $x$, no neighbour of $v$ is 0-dangerous (since $v$ is safe, $\pl^x_1(v)=0$). We now check that for every $r \in \{1, 2, \ldots, j\}$, the number $r$-dangerous neighbours of $v$ other than $x$, which we denote by $\dang(r)$, is smaller than $d/2j$. Every $1$-dangerous vertex has a cop as a neighbour. On the other hand, every cop is adjacent to at most $c$ neighbours of $v$, since otherwise we would have more than $c$ paths between this cop and $v$, contradicting Lemma ~\ref{l:elem}(ii). Moreover, by the assumption that $v$ is safe, we have $\pl_2^x(v)\le  \frac{d}{2cj}$. Combining all of these yields
$$
\dang(1) \le c \cdot \pl_{2}^x(v) \le c \Big(\frac{d}{2cj} \Big) = \frac {d}{2j}.
$$
For $r \in \{2, 3, \ldots, j-1\}$, we consider pairs $(y,w)$ where $y$ is an $r$-dangerous neighbour of $v$ and $w$ is a cop at distance $r$ from $y$. If $3\le r+1\le j$, then Lemma~\ref{l:elem}(ii) implies that there are at most $c$ paths between $v$ and $w$.  It follows that fewer than $c$ neighbours of $v$ are a distance $r$ from $w$. Estimating the number of pairs $(y,w)$ in two ways, we have
$$
\Big(\frac{d}{2cj} \Big)^{r-1}  \dang(r) \le c\cdot \pl_{r+1}^x(v) \le c\Big(\frac{d}{2cj} \Big)^r\,,
$$
and consequently $\dang(r)\le \frac{d}{2j}$.

Checking the desired bound for $\dang(j)$ is slightly more complicated. This time, a cop at distance $j+1$ from $v$ can contribute to the ``dangerousness'' of more than $c$ neighbours of $v$. However, the number of paths of length $j+1$ joining $v$ and $w$ is bounded from above by $\frac{6}{1-j \alpha}\frac{d^{j+1}}{n}$ (see Lemma~\ref{l:elem}(ii) and note that $d^{j+1} = n^{\alpha(j+1) + o(1)} \ge 7 n \log n$, since $\alpha(j+1) > 1$). Although we cannot control the number of cops in $N_{j+1}[v]$,  clearly $\pl_{j+1}^x(v)$ is bounded from above by $K=\frac {1-j\alpha}{12 (2c)^{j-1} j^j} \cdot \frac {n}{d}$, the total number of cops. Hence,
\begin{equation}\label{eq:calc_for_j}
\Big(\frac{d}{2cj} \Big)^{j-1}  \dang(j) \le \frac{6 d^{j+1}}{(1-j \alpha)n} \cdot K = c\Big(\frac{d}{2cj} \Big)^j
\end{equation}
and, as desired, $\dang(j)\le \frac{d}{2j}$. It follows that at most $d/2$ of neighbours of $v$ are $r$-dangerous for some $r=0,1,\dots, j$.

Now, it is time for the cops to make their move. Fortunately, only one cop may move, and this single cop can cause at most $c$ neighbours of $v$ to become $r$-dangerous for some $r < j$.
%(Recall that for $j$-dangerous vertices we considered all cops, including the active one.)  \bill{We never defined ``active'' -- and I think this reminder probably isn't needed anyway.}
Finally, we may use Lemma~\ref{l:elem}(i) and~(iii) to infer that there is a neighbour $y$ of $v$ that is not $r$-dangerous for any $r=0,1,\dots,j$, and such that $x$ does not belong to the $j$th neighbourhood of $y$ in $\G(n,p)\setminus \{v\}$. The vertex $y$ is safe; we move the robber there. This completes the proof of~(\ref{eq:lower_gnp1}).

Suppose now that $\alpha = 1/(j+1)$ for some $j \in \N$. The argument for this case is quite similar, so we only mention the differences. We consider three cases. First, suppose $d^{j+1} \ge 7 n \log n$. Since the number of paths of length $j+1$ from $v$ to a vertex $w$ is bounded from above by the same value, namely $\frac{6}{1-j \alpha}\frac{d^{j+1}}{n}$ (see Lemma~\ref{l:elem}(ii)), the calculations (and hence also the bound) are exactly the same.

Second, suppose $n/\log n \le d^{j+1} < 7 n \log n$. In this case, the number of paths of length $j+1$ from $v$ to $w$ is bounded above by $\frac{42}{1-j \alpha} \log n$ (as before, see Lemma~\ref{l:elem}(ii)). Hence, we must replace~(\ref{eq:calc_for_j}) by
$$
\Big(\frac{d}{2cj} \Big)^{j-1}  \dang(j) \le \frac{42}{1-j \alpha} \log n \cdot K = c\Big(\frac{d}{2cj} \Big)^j,
$$
provided that $K$ is adjusted to be $\frac{c(1-j\alpha)}{42 (2cj)^j} \cdot \frac{d^j}{\log n}$.

Third, suppose $d^{j+1} < n/\log n$. This time, the adjustments are slightly more complicated, since we must control the number of cops within distance $j+1$ of the robber. In particular, we need to find a neighbour $y$ of $v$ that is not $r$-dangerous for any $r=0,1,\ldots,j+1$ and such that $x$ does not belong to the $(j+1)$th neighbourhood of $y$ in $\G(n,p)\setminus \set{v}$. We adjust the definition of being ``safe'' as follows: $\pl^x_0(v)=0$, $\pl^x_1(v)=0$, $\pl_{i}^x(v) \le \Big(\frac{d}{2c(j+1)} \Big)^{i-1}$ for every $i=2,3,\dots,j$, and $\pl_{j+1}^x(v) \le \Big(\frac{d}{2c(j+1)} \Big)^{j} \frac {c(1-j\alpha)}{42 \log n}$. This assures that $\dang(j)$ is bounded as needed, that is,
$$
\Big(\frac{d}{2c(j+1)} \Big)^{j-1}  \dang(j) \le \frac{42}{1-j \alpha} \log n \cdot \pl_{j+1}^x(v) = c\Big(\frac{d}{2c(j+1)} \Big)^j.
$$
Finally, the number of $(j+1)$-dangerous neighbours of $v$ can be given by
$$
\Big(\frac{d}{2c(j+1)} \Big)^{j} \frac {c(1-j\alpha)}{42 \log n} \cdot \dang(j+1) \le \frac{84}{1-j \alpha} \frac {d^{j+2}\log n}{n} \cdot K,
$$
so $\dang(j+1)\le \frac{d}{2(j+1)}$, provided that $K$ is adjusted to be $\frac{c^2(1-j\alpha)^2}{3528 (2c(j+1))^{j+1}} \cdot \frac{n}{d \log^2 n}$.
\end{proof}

\section{Hypercubes}

In \cite{oo}, Offner and Ojakian provided asymptotic lower and upper bounds on $c_L(Q_n)$.  More precisely, they showed that $c_L(Q_n)=\Omega(2^{\sqrt{n}/20}) $ and $c_L(Q_n) = O(2^n \log n / n^{3/2})$.  In this section, we asymptotically improve the lower bound.  Our main result is the following:
\begin{theorem}\label{thm:hyp-lower}
For all $\eps> 0$, we have that
\[c_L(Q_n) = \Omega\of{\frac{2^n}{n^{7/2 + \eps}}}.\]
\end{theorem}
Thus, the upper and lower bounds on $c_L(Q_n)$ differ by only a polynomial factor.
\begin{proof}
We present a winning strategy for the robber provided that the number of cops is not too large. Let $\eps > 0$ be fixed, and suppose there are $k = k(\eps)$ cops (where $k$ will be chosen later).
We introduce a potential function that depends on each cop's distance to the robber. Let $N_i$ represent the number of cops at distance $i$ from the robber. With $\rho = \rho(n) = o(n)$, a function to be determined later (but such that $n/2 -\rho$ is an integer), we let
\[P = \sum_{i=1}^{n/2-\rho}N_iw_i\]
where, for $1\le i \le  \frac{n}{2} - \rho$, 
$$w_i = A\cdot\binom{n-1}{i}^{-1}\prod_{j=1}^{i}\of{1+\eps_j},\quad A= \frac{n-1}{1+\eps_1},$$
and $$\eps_i = \frac{4+\eps}{n-2i - 1} = o(1).$$
Note that this potential function ignores all cops at distance more than $n/2 - \rho$. We say that a cop at distance $i$ from the robber has {\em weight} $w_i$; this represents that cop's individual contribution toward the potential.  In particular, we have that $w_1 = 1$ and $w_2 = (1+o(1))2/n$.
If the cops can capture the robber on their turn, then immediately before the cops' turn we must have $P\ge 1$, since some cop must be at distance 1 from the robber.
Suppose that before the cops make their move, the potential function satisfies
\begin{equation}\label{Pinv}
P \le 1-\frac{3}{n};
\end{equation}
note that there cannot be a cop adjacent to the robber. Initially, we may assume that all cops start at the same vertex; the robber places himself at any vertex at distance at least $n/2 - \rho$ from the cops. Initially, $P = 0$, so (\ref{Pinv}) holds.  Our goal is to show that the robber can always enforce (\ref{Pinv}) right before the cops' move, from which it would follow that the robber can evade the cops indefinitely.

\smallskip
\textbf{Case 1}. Suppose that on the cops' turn, a cop moves to some vertex adjacent to the robber, creating a ``deadly'' neighbour for the robber.
%Then the increase in the potential function is
%\[w_1- w_2 = 1 - \frac{2}{n} + o\of{\frac{1}{n}}.\]
The robber's strategy is to move away from this ``deadly'' vertex, but to do so in a way that maintains the invariant~(\ref{Pinv}). To show that this is possible, we compute the expected change in the potential function if the robber were to choose his next position at random from among all neighbours other than the deadly one.

Consider a cop, $C$, at distance $i$ from the robber, where $2 \le i \le n/2-\rho$. Before the robber's move, $C$ has weight $w_i$. Let $w_C$ represent the expected weight of $C$ after the robber's move.  If $C$'s vertex and the deadly vertex differ on the deadly coordinate, then $w_C = \frac{i}{n-1}w_{i-1} + \frac{n-1-i}{n-1}w_{i+1}$, whereas if these coordinates coincide, then $w_C = \frac{i-1}{n-1}w_{i-1} + \frac{n-i}{n-1}w_{i+1}$. Since $w_{i-1} > w_{i+1}$, we may upper bound $w_C$ as follows:
\begin{align*}
w_C &\le \frac{i}{n-1}w_{i-1} + \frac{n-1-i}{n-1}w_{i+1}\\
&=\frac{i}{n-1}\cdot A \cdot\binom{n-1}{i-1}^{-1}\prod_{j=1}^{i-1}(1+\eps_j) + \frac{n-1-i}{n-1}\cdot A\cdot\binom{n-1}{i+1}^{-1}\prod_{j=1}^{i+1}(i+\eps_j) \\
&=\left(\frac{i}{n-1}(1+\eps_i)^{-1}\frac{(i-1)!(n-1-i+1)!}{(n-1)!}\,+ \right.\\
&\qquad \left.\frac{n-1-i}{n-1}(1+\eps_{i+1})\frac{(i+1)!(n-1-i-1)!}{(n-1)!}\right)\cdot A\cdot\prod_{j=1}^i(1+\eps_j)\\
&=w_i\of{\frac{n-i}{n-1}(1-\eps_i + o(\eps_i)) + \frac{i+1}{n-1}(1+\eps_i + o(\eps_i))}\\
&=w_i\of{1 + \frac{2}{n-1} - (1+o(1))\eps_i\of{\frac{n-i}{n-1} - \frac{i+1}{n-1}}}\\
&=w_i\of{1 + \frac{2}{n-1} - (1+o(1))\frac{4+\eps}{n-2i -1}\cdot\frac{n-2i-1}{n-1}} \le w_i\of{1 - \frac{2+\eps / 2}{n}}.
\end{align*}
Hence, after the robber's move, the expected sum of the weights of such cops has decreased by a multiplicative factor of at least $\of{1 - \frac{2+\eps/2}{n}}$, making it at most
\begin{equation}\label{eq:closeguys}
\of{1-\frac{3}{n}}\cdot\of{1-\frac{2+\eps/2}{n}}.
\end{equation}
 In addition, the cop that moved to the neighbourhood of the robber would again be at distance 2, making her weight
\begin{equation}\label{eq:dangguy}
w_2 = (1+o(1))\frac{2}{n}.
\end{equation} 
 It might also be that after the robber's move, some cops that were previously at distance $n/2-\rho+1$ from the robber are now at distance $n/2 -\rho$. 
By limiting the total number of cops, we may show that the total weight of cops at distance $n/2 - \rho$ is always negligible; that is, always less than, say, $\frac{\eps/4}{n}$.
The weight of a single cop at this distance is
\begin{equation}\label{eq:12}
w_{n/2 -\rho} = (1+o(1))n\cdot \binom{n-1}{n/2 - \rho}^{-1}\prod_{i=1}^{n/2 -\rho}\of{1 + \frac{4+\eps}{n-2i-1}}.
\end{equation}
We bound the product term in~(\ref{eq:12}) by
\begin{align*}
\prod_{i=1}^{n/2 - \rho}\of{1 + \frac{4+\eps}{n-2i-1}} &\le \exp\of{\sum_{i=1}^{n/2 - \rho}\frac{4+\eps}{n-2i-1}} \\
& = \exp\of{\frac{4+\eps}{2}\sum_{i=\rho}^{n/2}\frac{1}{i} + O(1)} \\
& = \exp\of{\frac{4+\eps}{2}\of{\log (n/2) - \log \rho + O(1)}}\\
&= O\of{\of{\frac{n}{\rho}}^{2+\eps/2}}.
\end{align*}
To bound the binomial term, we note that $\binom{n-1}{n/2-\rho} = \Theta\of{\binom{n}{n/2 -\rho}}$ and approximate: 
\begin{align*}
\binom{n}{n/2 - \rho}&=\frac{n!}{(n/2 - \rho)!(n/2 + \rho)!}\\
&=\frac{\sqrt{2\pi n}\of{\frac{n}{e}}^{n}}{\sqrt{2\pi (n/2 -\rho)}\of{\frac{n/2-\rho}{e}}^{n/2 -\rho} \sqrt{2\pi (n/2 +\rho)}\of{\frac{n/2+\rho}{e}}^{n/2 +\rho}   }(1+o(1))\\
&=\Theta\of{\frac{2^n}{\sqrt{n}}}\cdot \of{1 - \frac{2\rho}{n}}^{-\frac{n}{2} + \rho}\of{1+\frac{2\rho}{n}}^{-\frac{n}{2} - \rho} \\
&=\Theta\of{\frac{2^n}{\sqrt{n}}}\cdot \exp\of{-(1+o(1))\frac{2\rho}{n}\of{-\frac{n}{2} + \rho + \frac{n}{2} + \rho}} \\
&=\Theta\of{\frac{2^n}{\sqrt{n}}}\cdot \exp\of{-(1+o(1))\frac{4\rho^2}{n}}.
\end{align*}
Now take $\rho(n)$ to be minimal such that $\rho \ge \sqrt{n}$ and $n/2 - \rho$ is an integer.  Then we have that
\[w_{n/2 -\rho} = \Theta\of{n\cdot\frac{\sqrt{n}}{2^n}\cdot n^{1+\eps/4}} = \Theta\of{\frac{n^{5/2 + \eps/4}}{2^n}}.\]
So if we bound the total number of cops by $O(2^n / n^{7/2 + \eps})$, then we have that the total weight of cops at distance $n/2-\rho$ is at most
\begin{equation}\label{eq:newguys}
O\of{\frac{2^n}{n^{7/2 + \eps}}\cdot \frac{n^{5/2 + \eps/4}}{2^n}} < \frac{\eps/4}{n}.
\end{equation}
Thus, after the robber's random move, combining estimates \eqref{eq:closeguys}, \eqref{eq:dangguy} and \eqref{eq:newguys}, we can upper bound the total expected weight by 
\[\of{1-\frac{3}{n}}\cdot\of{1-\frac{2+\eps/2}{n}} + (1+o(1))\frac{2}{n} + \frac{\eps/4}{n}  \le 1 - \frac{3}{n} . \]
Hence some deterministic move produces a potential at least as low as the expectation, so the robber may maintain the invariant, as desired.

\smallskip
\textbf{Case 2}. Suppose now that a cop moves to a vertex at distance $i \ge 2$ from the robber. The resulting increase in the potential function is at most $2/n + o(1/n)$, so the new potential function has value at most $1 - 1/n + o(1/n)$.  Now, by the calculations from Case 1, the robber can move so that the total weight of all cops at distances $2$ through $n/2 - \rho$ decreases by a multiplicative factor of $(1- \frac{2+\eps}{n})$. Hence, after the robber's move, the potential is at most 
\[\of{1 - \frac{1}{n} + o\of{\frac{1}{n}}}\cdot\of{1 - \frac{2 + \eps/2}{n}} + \frac{\eps/4}{n} \le 1 - \frac{3}{n},\]
where we have once again taken into account the possibility of new cops at distance $n/2 - \rho$.
\end{proof}

\section{Graphs on surfaces}

The {\em genus} of a graph $G$ is the minimum genus of an orientable surface on which $G$ can be embedded without edge crossings. Graphs with genus 0 are the planar graphs, and it was shown in \cite{af} that planar graphs have cop number at most $3$. If $G$ has genus $g$, then it was proved in \cite{schr} that $c(G) \le \lfloor \frac{3g}{2}\rfloor +3.$ In the same paper, it was conjectured that $c(G)\le g+3.$

We conclude the paper with a straightforward asymptotic upper bound on $c_L$ for graphs with genus $g$.  We use a well-known separator result due to Gilbert, Hutchinson, and Tarjan~\cite{ght}.

\begin{theorem}[\cite{ght}]\label{thm:ght_sep}
Every $n$-vertex graph of genus $g$ contains set of at most $6\sqrt{gn} + 2\sqrt{2n} + 1$ vertices whose removal leaves a graph in which no component has more than $2n/3$ vertices.
\end{theorem}

We obtain our bound on $c_L$ as a direct consequence of Theorem~\ref{thm:ght_sep}.

\begin{theorem}\label{thm:bded_genus}
For every $n$-vertex graph $G$ of genus $g$ we have $c_L(G) \le 60\sqrt{gn} + 20\sqrt{2n}= O(\sqrt{gn})$. 
\end{theorem}
\begin{proof}
Let $K = 60\sqrt{gn} + 20\sqrt{2n} = O(\sqrt{gn})$; we use induction on $n$ to prove that $c_L(G) \le K$.  When $n = 1$, we have that $c_L(G) = 1$, so the bound holds.  Assume $n \ge 2$, and suppose that the bound holds for all graphs on fewer than $n$ vertices.

By Theorem~\ref{thm:ght_sep}, $G$ contains some separating set $S$ of cardinality at most $6\sqrt{gn} + 2\sqrt{2n} + 1$, such that each component of $G - S$ has at most $2n/3$ vertices.  The cops play as follows.  First, one cop occupies each vertex of $S$.  If the robber has not yet been captured, then he must inhabit some component $X$ of $G - S$.  The cops currently occupying vertices of $S$ remain in place for the duration of the game; consequently, the robber cannot leave $X$ without being captured.  The remaining $K - \size{S}$ cops now move to $X$ and, subsequently, attempt to capture the robber while remaining within $X$.  By choice of $X$ and the induction hypothesis, these cops may capture the robber so long as
$$K - \size{S} \ge 60\sqrt{g\frac{2n}{3}} + 20\sqrt{2\frac{2n}{3}}.$$
Since $\size{S} \le 6\sqrt{gn} + 2\sqrt{2n} + 1$, it suffices to show that
$$K \ge 60\sqrt{g\frac{2n}{3}} + 20\sqrt{2\frac{2n}{3}} + 6\sqrt{gn} + 2\sqrt{2n} + 1.$$
However,
\begin{align*}
60\sqrt{g\frac{2n}{3}} + 20\sqrt{2\frac{2n}{3}} + 6\sqrt{gn} + 2\sqrt{2n} + 1
   &= \left (60\sqrt{\frac{2}{3}}+6\right )\sqrt{gn} + \left (20\sqrt{\frac{2}{3}}+2\right )\sqrt{2n} + 1\\
   &< 55\sqrt{gn} + 19\sqrt{2n} + 1 < K,
\end{align*}
which completes the proof.
\end{proof}

It is not known whether the bounds in Theorem~\ref{thm:bded_genus} are asymptotically tight, even in the case of planar graphs.  In fact, we are not presently aware of any families of planar graphs on which the lazy cop number grows as an unbounded function.

\end{document}